\documentclass[reqno]{amsart}

\usepackage{mathrsfs}
\usepackage{amssymb}
\usepackage{amsfonts}
\usepackage{amsmath}
\usepackage{txfonts}
\usepackage{dsfont}
\usepackage{amscd}
\usepackage[all]{xy}
\usepackage{makecell}
\usepackage[colorlinks, linkcolor=blue, anchorcolor=blue, citecolor=blue]{hyperref}

\title{Invariant Einstein  metrics on basic classical  Lie supergroups}
\author{Huihui An,  Zaili Yan$^{*}$ \and Shaoxiang Zhang}
\address[Huihui An]{ School of Mathematics, Liaoning Normal
University, Dalian, Liaoning Province, 116029, People's Republic of China}
\email[]{anhh@lnnu.edu.cn }
\address[Zaili Yan]{School of Mathematics and Statistics, Ningbo University, Ningbo, Zhejiang Province, 315211,  People's Republic of China}
\email[]{yanzaili@nbu.edu.cn}
\address[Shaoxiang Zhang]{College of Mathematics and Systems Science, Shandong University of Science and Technology, Qingdao 266590, People's Republic of China}
\email[]{zhangshaoxiang93@163.com}
\thanks{H. An is supported by special fund for basic scientific research expenses of universities in Liaoning Province (LJ212410165012). $^{*}$Z. Yan is the corresponding author and is supported by Zhejiang Provincial Natural Science Foundation of China under Grant No. LMS25A010010. S. Zhang is partially supported by National Natural Science Foundation of China (No.12201358), Science and Technology Support Plan for Youth Innovation of Colleges and Universities of Shandong Province of China (No. 2023KJ090).}
\date{}

\newtheorem{thm}{Theorem}[section]
\newtheorem{prop}[thm]{Proposition}
\newtheorem{lem}[thm]{Lemma}
\newtheorem{cor}[thm]{Corollary}
\theoremstyle{definition}
\newtheorem{defn}[thm]{Definition}

\newtheorem{rem}[thm]{Remark}

\begin{document}

\maketitle
\begin{abstract}
This paper presents a systematic study of invariant Einstein metrics on basic classical Lie supergroups, whose Lie superalgebras belong to the Kac's classification of finite dimensional classical simple Lie superalgebras over $\mathbb{R}$. We consider a natural family of left invariant metrics parameterized by scaling factors on the simple and Abelian components of the reductive even part, using the canonical bi-invariant bilinear form. Explicit expressions for the Levi-Civita connection and Ricci tensor are derived, and the Einstein condition is reduced to a solvable algebraic system.
Our main result shows that, except for the cases of $\mathbf{A}(m,n)$ with $m\neq n$, $\mathbf{F}(4)$,  and their real forms, every  real basic classical Lie superalgebra
admits at least  two distinct Einstein metrics.
Notably, for $\mathbf{D}(n+1,n)$ and $\mathbf{D}(2,1;\alpha)$, we obtain both Ricci flat and non Ricci flat Einstein metrics, a phenomenon not observed in the non-super setting.



\end{abstract}

\section{Introduction}
The investigation of Einstein metrics, defined by the condition $\mathrm{ric}=c\langle\cdot, \cdot\rangle$ for some constant $c\in\mathbb{R}$, constitutes a profound topic in differential geometry with deep ties to physics, particularly general relativity and supergravity. In the homogeneous setting, especially on Lie groups and their generalizations, the existence and classification of invariant Einstein metrics have been extensively studied,
 fostering rich interactions among geometry, algebra, and analysis \cite{bk23,bl23,bwz04,heb98,Jab23,lau10}.
 The natural extension of this framework to the supersymmetric context: Lie supergroups and homogeneous supermanifolds,
  has opened new avenues of research, blending techniques from representation theory of Lie superalgebras, differential supergeometry, and mathematical physics.
  More recently,  Zhang, Gould, Pulemotov and Rasmussen in \cite{ZGPR24} initiated the study of the Einstein equation on homogeneous supermanifolds, deriving explicit formulas for the Ricci curvature and providing a complete classification of Einstein metrics on certain flag supermanifolds.

In the spirit of \cite{dz79} on left invariant naturally reductive Einstein metrics on compact Lie groups, this paper presents a systematic investigation of invariant Einstein metrics on basic classical Lie supergroups.
 These are Lie supergroups whose Lie superalgebras belong to the celebrated classification by Kac \cite{Kac} of finite dimensional classical simple Lie superalgebras over $\mathbb{R}$. Specifically, we focus on the families $\mathbf{A}(m,n)$, $\mathbf{B}(m,n)$, $\mathbf{C}(n)$, $\mathbf{D}(m,n)$, $\mathbf{D}(2,1;\alpha)$, $\mathbf{F}(4)$, and $\mathbf{G}(3)$, along with their real forms. A defining feature  of these Lie superalgebras is the existence of a non-degenerate even supersymmetric bi-invariant bilinear form $\mathbf{B}$. This form generalizes the Killing form and serves as the foundational structure for constructing invariant Einstein metrics on the corresponding Lie supergroups.
We refer the readers to \cite{ABB09,BB99,Benayadi00} for related results on non-simple quadratic Lie superalgebras.

Let $\mathfrak{g}=\mathfrak{g}_{\bar{0}}+\mathfrak{g}_{\bar{1}}$
be  a real Lie superalgebra endowed with a bi-invariant metric $\mathbf{B}$.
Assume that $\mathfrak{g}_{\bar{0}}$ is reductive, so it decomposes into an Abelian part $\mathfrak{k}_{0}$ and simple ideals $\mathfrak{k}_{1},\ldots,\mathfrak{k}_{s}$:
\begin{equation*}
  \mathfrak{g}_{\bar{0}}=\mathfrak{k}_{0}\oplus\mathfrak{k}_{1}\oplus\cdots\oplus\mathfrak{k}_{s}.
\end{equation*}
On the associated Lie supergroup $G$, we consider a family of left invariant metrics of the form:
\begin{equation}\label{Einstein-metric-form}
\langle\cdot,\cdot\rangle
 =x_{0}\mathbf{B}|_{\mathfrak{k}_{0}}+x_{1}\mathbf{B}|_{\mathfrak{k}_{1}}+\cdots+x_{s}\mathbf{B}|_{\mathfrak{k}_{s}}
 +\mathbf{B}|_{\mathfrak{g}_{\bar{1}}}, \quad x_{i}\in \mathbb{R}^{\ast}.
\end{equation}
This family is natural in that it respects the decomposition of $\mathfrak{g}$ and employs the canonical bilinear form $\mathbf{B}$. The parameters $x_{i}$ allow independent scaling of the metric on each component.
We derive explicit formulas for the Levi-Civita connection and Ricci tensor (see Lemma \ref{3-lem-levi-civita} and Proposition \ref{3-prop-ric-tensor}),
along with a necessary and sufficient condition for $\langle\cdot,\cdot\rangle$ to be Einstein. In particular, when $\mathfrak{g}$ is basic classical, $\mathbf{B}=\mathbf{K}$, the Killing form of $\mathfrak{g}$, and $\mathbf{K}$ is non-degenerate,  the Einstein condition $\mathrm{ric}=c\langle\cdot, \cdot\rangle$ reduces to the following system of algebraic equations in the parameters $x_i$ and $c$ (Theorem \ref{4-thm-B=K}):
\begin{eqnarray*}
\left\{
\begin{aligned}
&c=-\frac{x_{0}}{4},\quad \mathrm{if}\, \mathfrak{k}_{0}\neq0, \\
&\frac{1}{4}(l_{i}x_{i}^{2}-1)=c (1-l_{i})x_{i},\quad 1\leq i\leq s,\\
&-\frac{\mathrm{dim}\,\mathfrak{k}_{0}}{\mathrm{dim}\,\mathfrak{g}_{\bar{1}}}x_{0}
+\sum\limits_{i=1}^{s}\frac{l_{i}\mathrm{dim}\,\mathfrak{k}_{i}}{(1-l_{i})\mathrm{dim}\,\mathfrak{g}_{\bar{1}}}x_{i}=2c+1,
\end{aligned}\right.
\end{eqnarray*}
where $l_{i}$ is the index of the representation of $\mathfrak{k}_{i}$
on $\mathfrak{g}_{\bar{1}}$.

By analyzing the Einstein equation  for each type of basic classical Lie superalgebra, we prove our  main theorem.
\begin{thm}\label{1-thm-main}
Except for the cases of $\mathbf{A}(m,n)$ with $m\neq n$, $\mathbf{F}(4)$,  and their real forms, every  real basic classical Lie superalgebra
admits at least  two distinct Einstein metrics of the form \eqref{Einstein-metric-form}.
\end{thm}
\begin{rem}
For both $\mathbf{D}(n+1,n)$ and $\mathbf{D}(2,1;\alpha)$, as well as their real forms, we obtain Einstein metrics with both vanishing and non-vanishing scalar curvature ($c=0$ and $c\neq0$). This phenomenon may not occur in the non-super setting, as it remains unknown whether a simple Lie algebra admits a Ricci flat metric.
\end{rem}

The paper is organized as follows. In Section \ref{sect2}, we review necessary background on Lie superalgebras. In Section \ref{sect3}, we derive the general Einstein equations for the chosen family of metrics. Section \ref{sect4} is devoted to the classification of solutions for each type of basic classical Lie superalgebra.

\section{Lie superalgebras}\label{sect2}

A Lie superalgebra is a vector superspace $\mathfrak{g}=\mathfrak{g}_{\bar{0}}+\mathfrak{g}_{\bar{1}}$ (over $\mathbb{R}$ or $\mathbb{C}$) equipped with a Lie superbracket $[\cdot,\cdot]$ satisfying
\begin{eqnarray*}
\left\{
\begin{aligned}
&[X,Y]=-(-1)^{[X][Y]}[Y,X],\\
&[X,[Y,Z]]=[[X,Y],Z]+(-1)^{[X][Y]}[Y,[X,Z]],\quad \forall X,Y,Z\in\mathfrak{g},
\end{aligned}
\right.
\end{eqnarray*}
where $[X]\in \mathbb{Z}_{2}$ denotes the parity of a homogeneous element $X\in\mathfrak{g}$.
Throughout this paper, we assume $\mathfrak{g}$ is non-trivial, i.e., $\mathfrak{g}_{\bar{1}}\neq 0$.
 All vectors $X\in \mathfrak{g}$ in an expression are assumed to be homogeneous, i.e., $X\in\mathfrak{g}_{\bar{0}}\cup \mathfrak{g}_{\bar{1}}$, and then the expression extends to the other elements by linearity.

A finite dimensional Lie superalgebra $\mathfrak{g}=\mathfrak{g}_{\bar{0}}+\mathfrak{g}_{\bar{1}}$ is called classical if $\mathfrak{g}$ is simple and the representation of $\mathfrak{g}_{\bar{0}}$ on $\mathfrak{g}_{\bar{1}}$ is completely reducible. A bilinear form $f$ on $\mathfrak{g}$ is called even if $f(\mathfrak{g}_{\bar{0}},\mathfrak{g}_{\bar{1}})=0$, supersymmetric if $f(X,Y)=(-1)^{[X][Y]}f(Y,X)$, and bi-invariant if $f([X,Y],Z)=f(X,[Y,Z])$.
We say  a classical Lie superalgebra is basic if it admits a non-degenerate
even supersymmetric bi-invariant bilinear form.

The Killing form $\mathbf{K}$ of a finite dimensional Lie superalgebra $\mathfrak{g}$ is defined by $\mathbf{K}(X,Y)=\mathrm{str}(\mathrm{ad}\,X\circ\mathrm{ad}\,Y)$ for all $X,Y\in\mathfrak{g}$,
 where $\mathrm{ad}\,X(Y):=[X,Y]$ and  $\mathrm{str}$ denotes the supertrace.
 For a  matrix
$\left(\begin{array}{cc}
A&B\\
 C&D
\end{array}\right)$
 in $\mathrm{End}(\mathfrak{g})$,
the supertrace is given by $\mathrm{tr}\,A-\mathrm{tr}\,D$. Note that the Killing form $\mathbf{K}$ is even, supersymmetric,  bi-invariant,  but not necessarily non-degenerate. The following result is due to Kac \cite{Kac}.
\begin{thm}
A complex basic classical Lie superalgebra  is isomorphic to one of $\mathbf{A}(m,n)$, $\mathbf{B}(m,n)$, $\mathbf{C}(n)$, $\mathbf{D}(m,n)$, $\mathbf{D}(2,1;\alpha)$, $\mathbf{F}(4)$ or $\mathbf{G}(3)$.
Every real basic classical Lie superalgebra is isomorphic  either to a complex basic classical Lie superalgebra (regarded as a real superalgebra) or to  a real form of such a  complex superalgebra.
In particular, the Killing forms of $\mathbf{A}(n,n)$, $\mathbf{D}(n+1,n)$ and $\mathbf{D}(2,1;\alpha)$ are identically zero.
\end{thm}

We now recall the notion of the index of a representation of a simple Lie algebra. Let $\mathfrak{k}$ be a finite dimensional simple Lie algebra over $\mathbb{R}$ or $\mathbb{C}$,  and let  $\rho:\mathfrak{k}\rightarrow \mathfrak{gl}(V)$ be a finite dimensional representation. By  Schur's lemma, we have
\begin{eqnarray*}
\mathrm{tr}\,\rho(X)\rho(Y)=l_{\rho}\mathrm{tr}\,(\mathrm{ad}\,X\circ\mathrm{ad}\,Y),
\end{eqnarray*}
where $l_{\rho}\in \mathbb{R}$ is  independent of $X$ and $Y$,  called the index of $\rho$. If $\rho=\rho_{1}+\rho_{2}$ is a direct sum, then $l_{\rho}=l_{\rho_{1}}+l_{\rho_{2}}$.
Also $l_{\rho}=l_{\rho^{*}}$, where $\rho^{*}$ is  the contragredient representation of  $\rho$.
 Table 1 lists the indices of standard representations of  simple Lie algebras,
 see \cite[Table III]{Kac}.
\begin{table}[htbp]\label{table1}
\caption{Indices of standard representations}
\centering
\setlength\tabcolsep{12 pt}
\renewcommand{\arraystretch}{1.3}
\begin{tabular}{|c|c|c|}
\hline
$\rho$                &  $\mathrm{dim}\,V$    &  $l_{\rho}$\\
\hline
$\mathfrak{sl}(n)$    &    $n$                &  $\frac{1}{2n}$\\
\hline
$\mathfrak{so}(n)$    &    $n$                &  $\frac{1}{n-2}$\\
\hline
$\mathfrak{spin}(7)$  &    $8$                &  $\frac{1}{5}$\\
\hline
$\mathfrak{sp}(n)$    &    $n$                &  $\frac{1}{n+2}$\\
\hline
$\mathfrak{g}_{2}$    &    $7$                &  $\frac{1}{4}$\\
\hline
\end{tabular}
\end{table}
\\Here $\mathfrak{spin}(7)$ stands for the irreducible spinor representation of $\mathfrak{so}(7)$ and $\mathfrak{g}_{2}$ refers to the simplest representation of the 14-dimensional simple Lie algebra $\mathfrak{g}_{2}$.
\begin{prop}\label{2-prop-index}
Let $\mathfrak{g}=\mathfrak{g}_{\bar{0}}+\mathfrak{g}_{\bar{1}}$ be a real Lie superalgebra with $\mathfrak{g}_{\bar{0}}$ reductive, and let $\mathfrak{g}\otimes\mathbb{C}=\mathfrak{g}_{\bar{0}}\otimes\mathbb{C}+\mathfrak{g}_{\bar{1}}\otimes\mathbb{C}$ be its complexification.  Let $\mathfrak{k}\subset \mathfrak{g}_{\bar{0}}$ be a real simple ideal.

(i) Denote by $l_{\mathfrak{k}}$ and $l_{\mathfrak{k}\otimes\mathbb{C}}$  the indices of the representations of $\mathfrak{k}$ and $\mathfrak{k}\otimes\mathbb{C}$ on $\mathfrak{g}_{\bar{1}}$ and $\mathfrak{g}_{\bar{1}}\otimes\mathbb{C}$ respectively, then $l_{\mathfrak{k}}=l_{\mathfrak{k}\otimes\mathbb{C}}$.

(ii) If $\mathfrak{g}$ is a complex Lie superalgebra regarded as a real Lie superalgebra, and $l_{\mathfrak{k}}^{\mathbb{C}}$ is the index of the representation of the complex simple Lie algebra $\mathfrak{k}$ on $\mathfrak{g}_{\bar{1}}$ (viewed as a complex vector space), then $l_{\mathfrak{k}}=l_{\mathfrak{k}}^{\mathbb{C}}$.
\end{prop}

\begin{proof}
(\romannumeral1) Let $\mathbf{K}_{\mathfrak{g}\otimes\mathbb{C}}$, $\mathbf{K}_{\mathfrak{k}}$ and $\mathbf{K}_{\mathfrak{k}\otimes\mathbb{C}}$ denote the Killing forms of $\mathfrak{g}\otimes\mathbb{C}$, $\mathfrak{k}$ and $\mathfrak{k}\otimes\mathbb{C}$, respectively. Then we have
\begin{eqnarray*}
\mathbf{K}(X,Y)=\mathbf{K}_{\mathfrak{g}\otimes\mathbb{C}}(X,Y), \quad  \mathbf{K}_{\mathfrak{k}}(X,Y)=\mathbf{K}_{\mathfrak{k}\otimes\mathbb{C}}(X,Y), \quad \forall X,Y\in \mathfrak{k}.
\end{eqnarray*}
Hence $(1-l_{\mathfrak{k}})\mathbf{K}_{\mathfrak{k}}(X,Y)
=(1-l_{\mathfrak{k}\otimes\mathbb{C}})\mathbf{K}_{\mathfrak{k}\otimes\mathbb{C}}(X,Y)$, and thus $l_{\mathfrak{k}}=l_{\mathfrak{k}\otimes\mathbb{C}}$.

(\romannumeral2) Let $\mathbf{K}_{\mathfrak{g}}^{\mathbb{C}}$ and $\mathbf{K}_{\mathfrak{k}}^{\mathbb{C}}$ denote the Killing forms of the  complex Lie superalgebra $\mathfrak{g}$ and the complex Lie algebra $\mathfrak{k}$, respectively. Then
\begin{eqnarray*}
\mathbf{K}(X,Y)=2\mathrm{Re}(\mathbf{K}_{\mathfrak{g}}^{\mathbb{C}}(X,Y)),\quad
\mathbf{K}_{\mathfrak{k}}(X,Y)=2\mathrm{Re}(\mathbf{K}_{\mathfrak{k}}^{\mathbb{C}}(X,Y)),
\quad \forall X,Y\in\mathfrak{k}.
\end{eqnarray*}
Therefore
\begin{eqnarray*}
(1-l_{\mathfrak{k}})\mathbf{K}_{\mathfrak{k}}(X,Y)=2\mathrm{Re}\Big{(}(1-l_{\mathfrak{k}}^{\mathbb{C}})\mathbf{K}_{\mathfrak{k}}^{\mathbb{C}}(X,Y)\Big{)},
\end{eqnarray*}
which implies $l_{\mathfrak{k}}=l_{\mathfrak{k}}^{\mathbb{C}}$.
\end{proof}

In Table 2,  we list  the indices $l_{i}$ of the representation of each simple ideal $\mathfrak{k}_{i}$ in $\mathfrak{g}_{\bar{0}}$ on $\mathfrak{g}_{\bar{1}}$, for $\mathfrak{g}$  a complex basic classical Lie superalgebra with  $\mathfrak{g}_{\bar{0}}=\mathfrak{k}_{0}\oplus\mathfrak{k}_{1}\oplus\cdots\oplus\mathfrak{k}_{s}$,
where  $\mathfrak{k}_{0}$ is Abelian and $\mathfrak{k}_{i}$ ($1\leq i\leq s$) are simple ideals.
 By Proposition \ref{2-prop-index}, we can readily obtain the corresponding indices for real basic classical Lie superalgebras.
\begin{table}[htbp]
\caption{Indices of basic classical Lie superalgebras }
\centering
\label{T2}
\small{
\setlength{\tabcolsep}{1mm}{
\begin{tabular}{|c|c|c|c|}
\hline
$\mathfrak{g}$     &  $\mathfrak{g}_{\bar{0}}$    &  $\mathfrak{g}_{\bar{0}}|\mathfrak{g}_{\bar{1}}$  & $l_{i}$\\
\hline
$\mathbf{A}(m,n)$  &  $\mathbb{C}\oplus \mathbf{A}_{m}\oplus \mathbf{A}_{n}$  &  \makecell{$\mathbb{C}\otimes\mathfrak{sl}(m+1)\otimes\mathfrak{sl}(n+1)$\\
                                                                                           $+(\mathbb{C}\otimes\mathfrak{sl}(m+1)\otimes\mathfrak{sl}(n+1))^{\ast}$}
                                                                              &  \makecell{$l_{0}=\infty$\\
                                                                                           $l_{1}=\frac{n+1}{m+1}$\\
                                                                                           $l_{2}=\frac{m+1}{n+1}$}
\\
\hline
$\mathbf{A}(n,n)$  &  $\mathbf{A}_{n}\oplus \mathbf{A}_{n}$                   &  \makecell{$\mathfrak{sl}(n+1)\otimes\mathfrak{sl}(n+1)$\\
                                                                                           $+(\mathfrak{sl}(n+1)\otimes\mathfrak{sl}(n+1))^{\ast}$}
                                                                              &  $l_{1}=l_{2}=1$
\\
\hline
$\mathbf{B}(m,n)$  &  $\mathbf{B}_{m}\oplus \mathbf{C}_{n}$  & $\mathfrak{so}(2m+1)\otimes\mathfrak{sp}(2n)$
                                                                              &  \makecell{$l_{1}=\frac{2n}{2m-1}$\\
                                                                                           $l_{2}=\frac{2m+1}{2n+2}$}
\\
\hline
$\mathbf{C}(n)$  &  $\mathbb{C}\oplus \mathbf{C}_{n-1}$                       &  $\mathbb{C}\otimes\mathfrak{sp}(2n-2)+(\mathbb{C}\otimes\mathfrak{sp}(2n-2))^{\ast}$
                                                                              &  \makecell{$l_{0}=\infty$\\
                                                                                           $l_{1}=\frac{1}{n}$}
\\
\hline
$\mathbf{D}(m,n)$  &  $\mathbf{D}_{m}\oplus \mathbf{C}_{n}$  & $\mathfrak{so}(2m)\otimes\mathfrak{sp}(2n)$
                                                                              &  \makecell{$l_{1}=\frac{n}{m-1}$\\
                                                                                           $l_{2}=\frac{m}{n+1}$}
\\
\hline
$\mathbf{F}(4)$  &  $\mathbf{B}_{3}\oplus \mathbf{A}_{1}$  & $\mathfrak{spin}(7)\otimes\mathfrak{sl}(2)$
                                                                              &  \makecell{$l_{1}=\frac{2}{5}$\\
                                                                                           $l_{2}=2$}
\\
\hline
$\mathbf{G}(3)$  &  $\mathfrak{g}_{2}\oplus \mathbf{A}_{1}$  & $\mathfrak{g}_{2}\otimes\mathfrak{sl}(2)$
                                                                              &  \makecell{$l_{1}=\frac{1}{2}$\\
                                                                                           $l_{2}=\frac{7}{4}$}
\\
\hline
\end{tabular}}}
\end{table}

\section{Einstein equations}\label{sect3}

Let $G$ be a Lie supergroup with Lie superalgebra $\mathfrak{g}$ over $\mathbb{R}$, consisting of left invariant vector fields on $G$.
A left invariant metric $\langle\cdot,\cdot\rangle$ on $G$ is a non-degenerate even supersymmetric bilinear form on $\mathfrak{g}$ (see \cite[Theorem 3]{Goer08}). The Levi-Civita connection $\nabla$ of $(\mathfrak{g},\langle\cdot,\cdot\rangle)$ satisfies
\begin{eqnarray*}
\left\{
\begin{aligned}
&\langle\nabla_{X}Y,Z\rangle+(-1)^{[X][Y]}\langle Y,\nabla_{X}Z\rangle=0,\\
&\nabla_{X}Y-(-1)^{[X][Y]}\nabla_{Y}X=[X,Y],\quad \forall X,Y,Z\in\mathfrak{g}.
\end{aligned}\right.
\end{eqnarray*}
\begin{lem}[\cite{Goer08,ZGPR24}]\label{3-lem-civita}
For all $ X,Y,Z\in\mathfrak{g}$, we have
\begin{eqnarray}\label{Eq1}
2\langle\nabla_{X}Y,Z\rangle=\langle[X,Y],Z\rangle-\langle X,[Y,Z]\rangle-(-1)^{[X][Y]}\langle Y,[X,Z]\rangle.
\end{eqnarray}
\end{lem}

It follows from Lemma \ref{3-lem-civita} that $\nabla_{X}(\cdot) \in \mathrm{End}(\mathfrak{g})_{[X]}$.
 The Ricci tensor $\mathrm{ric}$ of $(\mathfrak{g},\langle\cdot,\cdot\rangle)$ is defined by
\begin{eqnarray*}
\mathrm{ric}(X,Y)=\mathrm{str}(Z\mapsto R(Z,X)Y),
\end{eqnarray*}
where $R$ is the curvature tensor:
\begin{eqnarray*}
R(X,Y)Z:=\nabla_{X}\nabla_{Y}Z-(-1)^{[X][Y]}\nabla_{Y}\nabla_{X}Z-\nabla_{[X,Y]}Z.
\end{eqnarray*}
\begin{rem}
Our definition of the Ricci tensor $\mathrm{ric}$ differs from that in \cite[formula (3.6)]{ZGPR24} by a sign.
\end{rem}
\begin{defn}
The metric $\langle\cdot,\cdot\rangle$ on $G$ is called Einstein if $\mathrm{ric}=c\langle\cdot,\cdot\rangle$ for some $c\in\mathbb{R}$. It is called Ricci flat if  $c=0$.
\end{defn}

The Ricci tensor $\mathrm{ric}$ is an even supersymmetric bilinear form on $\mathfrak{g}$. Hence, if $\mathfrak{g}$ is basic classical, then any bi-invariant metric is Einstein by Schur's lemma \cite{Kac}.

Let $\mathfrak{g}=\mathfrak{g}_{\bar{0}}+\mathfrak{g}_{\bar{1}}$ be a real Lie superalgebra with $\mathfrak{g}_{\bar{0}}$ reductive,  so that $\mathfrak{g}_{\bar{0}}=\mathfrak{k}_{0}\oplus\mathfrak{k}_{1}\oplus\cdots\oplus\mathfrak{k}_{s}$, where $\mathfrak{k}_{0}$ is Abelian and the $\mathfrak{k}_{i}$ ($1\leq i\leq s$) are simple ideals.
Suppose $\mathbf{B}$ is a bi-invariant metric on $\mathfrak{g}$. Then
 $\mathbf{B}(\mathfrak{k}_{i},\mathfrak{k}_{j})=0$ for $i\neq j$,  and $\mathbf{B}|_{\mathfrak{k}_{i}}$ is non-degenerate.
Moreover,  there exist $b_{1},\ldots,b_{s}\in\mathbb{R}$ such that
$$\mathbf{B}|_{\mathfrak{k}_{i}}=b_{i}\mathbf{K}_{i}, \, \forall 1\leq i \leq s,$$
 where $\mathbf{K}_{i}$  is the Killing form of $\mathfrak{k}_{i}$.
 If $\mathbf{B}=\mathbf{K}$, then $b_{i}=1-l_{i}$ for $1\leq i \leq s$, where $l_{i}$ is the index of the representation of $\mathfrak{k}_{i}$ on $\mathfrak{g}_{\bar{1}}$.
 Let $\mathcal{C}_{\mathbf{B}|_{\mathfrak{k}_{i}}}=\sum\limits_{j}(\mathrm{ad}\,e_{j}\circ\mathrm{ad}\,e_{j}^{\ast})|_{\mathfrak{g}_{\bar{1}}}$ be the Casimir operator of $\mathfrak{k}_{i}$ on $\mathfrak{g}_{\bar{1}}$ with respect to $\mathbf{B}|_{\mathfrak{k}_{i}}$, where $\{e_{j}\}\subset \mathfrak{k}_{i}$ is a basis and $\{e_{j}^{\ast}\}$ is the dual basis, i.e., $\mathbf{B}(e_{j},e_{k}^{\ast})=\delta_{jk}$.
Then  $\mathcal{C}_{\mathbf{B}|_{\mathfrak{k}_{i}}}$ is symmetric with respect to $\mathbf{B}$, namely,
$\mathbf{B}(\mathcal{C}_{\mathbf{B}|_{\mathfrak{k}_{i}}}(X),Y)
=\mathbf{B}(X,\mathcal{C}_{\mathbf{B}|_{\mathfrak{k}_{i}}}(Y))$ for all $X,Y\in \mathfrak{g}_{\bar{1}}$.
In particular, by Schur's lemma,  if $V\subset \mathfrak{g}_{\bar{1}}$ is a $\mathfrak{k}_{i}$-irreducible subspace,  then
$$\mathcal{C}_{\mathbf{B}|_{\mathfrak{k}_{i}}}|_{V}
=\frac{l_{i}\mathrm{dim}\,\mathfrak{k}_{i}}{b_{i}\mathrm{dim}\,V}\mathrm{Id}|_{V}.$$

Consider the  family of metrics on  $\mathfrak{g}$ of the form \eqref{Einstein-metric-form}:
\begin{equation*}
\langle\cdot,\cdot\rangle
 =x_{0}\mathbf{B}|_{\mathfrak{k}_{0}}+x_{1}\mathbf{B}|_{\mathfrak{k}_{1}}+\cdots+x_{s}\mathbf{B}|_{\mathfrak{k}_{s}}
 +\mathbf{B}|_{\mathfrak{g}_{\bar{1}}}, \quad x_{i}\in \mathbb{R}^{\ast}.
\end{equation*}
 We first show that $\langle\cdot,\cdot\rangle$ is naturally reductive
 in the sense of \cite[Section 3.5]{ZGPR24}.
 \begin{prop}
 $\langle\cdot,\cdot\rangle$ is naturally reductive with respect to $\mathfrak{g}\oplus\mathfrak{g}_{\bar{0}}$.
 \end{prop}
 \begin{proof}
Let $\mathrm{diag}(\mathfrak{g}_{\bar{0}})=\{(X,X)|X\in\mathfrak{g}_{\bar{0}}\}$.
There exists a direct sum decomposition:
$$\mathfrak{g}\oplus\mathfrak{g}_{\bar{0}}
=\mathrm{diag}(\mathfrak{g}_{\bar{0}})+\mathfrak{m}
=\mathrm{diag}(\mathfrak{g}_{\bar{0}})
+\mathfrak{m}_{0}(t_{0})+\mathfrak{m}_{1}(t_{1})+\cdots+\mathfrak{m}_{s}(t_{s})+\mathfrak{m}_{\bar{1}},$$
where
$\mathfrak{m}_{i}(t_{i})=\{(t_{i}X,(t_{i}-1)X)|X\in\mathfrak{k}_{i}\}$,
 $t_{i}\in \mathbb{R}$ for $0\leq i\leq s$,
 $\mathfrak{m}_{\bar{1}}=\{(X,0)|X\in\mathfrak{g}_{\bar{1}}\}$.
Consider the linear isomorphism  $\mathfrak{m}\rightarrow\mathfrak{g}$ given by
\begin{eqnarray*}
\sum_{i=0}^{s}(t_{i}X_{i},(t_{i}-1)X_{i})+(X_{\bar{1}},0)\longleftrightarrow \sum_{i=0}^{s}X_{i}+X_{\bar{1}}, \quad X_{i}\in\mathfrak{k}_{i}, X_{\bar{1}}\in\mathfrak{g}_{\bar{1}}.
\end{eqnarray*}
The metric $\langle\cdot,\cdot\rangle$ on $\mathfrak{g}$ induces  an $\mathrm{ad}\,(\mathrm{diag}(\mathfrak{g}_{\bar{0}}))$-invariant metric $\langle\cdot,\cdot\rangle'$ on $\mathfrak{m}$ defined by
\begin{eqnarray*}
\left\{
\begin{aligned}
& \langle \mathfrak{m}_{i}(t_{i}),\mathfrak{m}_{j}(t_{j}) \rangle'=0, \  i\neq j,\\
& \langle \mathfrak{m}_{i}(t_{i}),\mathfrak{m}_{\bar{1}} \rangle'=0,\\
& \langle (t_{i}X_{i},(t_{i}-1)X_{i}), (t_{i}Y_{i},(t_{i}-1)Y_{i})\rangle'=\langle X_{i},Y_{i}\rangle=x_{i}\mathbf{B}(X_{i},Y_{i}),
                                                                               \ X_{i},Y_{i}\in\mathfrak{k}_{i}, 0\leq i\leq s,\\
& \langle (X,0),(Y,0)\rangle'=\langle X,Y\rangle=\mathbf{B}(X,Y), \ X,Y\in\mathfrak{g}_{\bar{1}}.
\end{aligned}\right.
\end{eqnarray*}

We now choose $t_{0},t_{1},\ldots,t_{s}$ such that  $\langle\cdot,\cdot\rangle'$ on $\mathfrak{m}$ is naturally reductive, i.e., $\langle [X,Y]_{\mathfrak{m}},Z\rangle'=\langle X,[Y,Z]_{\mathfrak{m}}\rangle'$ holds for all $X,Y,Z\in\mathfrak{m}$. Let
$X=\sum_{i=0}^{s}(t_{i}X_{i},(t_{i}-1)X_{i})+(X_{\bar{1}},0)$,
$Y=\sum_{i=0}^{s}(t_{i}Y_{i},(t_{i}-1)Y_{i})+(Y_{\bar{1}},0)$
and $Z=\sum_{i=0}^{s}(t_{i}Z_{i},(t_{i}-1)Z_{i})+(Z_{\bar{1}},0)$,
where $X_{i},Y_{i},Z_{i}\in\mathfrak{k}_{i}$, $X_{\bar{1}},Y_{\bar{1}},Z_{\bar{1}}\in\mathfrak{g}_{\bar{1}}$.
A direct computation yields
\begin{eqnarray*}
[X,Y]&=& \sum_{i=0}^{s}\left( t_{i}^{2}[X_{i},Y_{i}],(t_{i}-1)^{2}[X_{i},Y_{i}] \right)
         +\left( [X_{\bar{1}},Y_{\bar{1}}],0\right) \\
      && +\left(
         \left[\sum_{i=0}^{s}t_{i}X_{i},Y_{\bar{1}}\right] +\left[ X_{\bar{1}},\sum_{i=0}^{s}t_{i}Y_{i}\right], 0
         \right).
\end{eqnarray*}
Hence
\begin{eqnarray*}
[X,Y]_{\mathfrak{m}}&=& \sum_{i=0}^{s}(2t_{i}-1)\left( t_{i}[X_{i},Y_{i}],(t_{i}-1)[X_{i},Y_{i}]\right)\\
                     && +\sum_{i=0}^{s} \left( t_{i}[X_{\bar{1}},Y_{\bar{1}}]_{\mathfrak{k}_{i}},(t_{i}-1)[X_{\bar{1}},Y_{\bar{1}}]_{\mathfrak{k}_{i}}\right)\\
                     && +\left(
                         \left[ \sum_{i=0}^{s}t_{i}X_{i},Y_{\bar{1}}\right] +\left[ X_{\bar{1}},\sum_{i=0}^{s}t_{i}Y_{i}\right],0
                         \right),
\end{eqnarray*}
and therefore
\begin{eqnarray*}
\langle [X,Y]_{\mathfrak{m}},Z\rangle'&=&\sum_{i=0}^{s} x_{i}\mathbf{B}\left(
                                          (2t_{i}-1)[X_{i},Y_{i}]+[X_{\bar{1}},Y_{\bar{1}}]_{\mathfrak{k}_{i}},Z_{i}
                                          \right) \\
                                      && +\mathbf{B}\left(
                                           \left[ \sum_{i=0}^{s}t_{i}X_{i},Y_{\bar{1}}\right]+\left[ X_{\bar{1}},\sum_{i=0}^{s}t_{i}Y_{i}\right] ,Z_{\bar{1}}
                                          \right)\\
                                       &=&\sum_{i=0}^{s}\Big(
                                              x_{i}(2t_{i}-1)\mathbf{B}([X_{i},Y_{i}],Z_{i})+x_{i}\mathbf{B}([X_{\bar{1}},Y_{\bar{1}}],Z_{i}) \\
                                       &&+t_{i}\mathbf{B}([X_{i},Y_{\bar{1}}],Z_{\bar{1}})+t_{i}\mathbf{B}([X_{\bar{1}},Y_{i}],Z_{\bar{1}})
                                            \Big).
\end{eqnarray*}
Similarly, we obtain
\begin{eqnarray*}
\langle X, [Y,Z]_{\mathfrak{m}}\rangle'&=&\sum_{i=0}^{s}\Big(
                                          x_{i}(2t_{i}-1)\mathbf{B}(X_{i},[Y_{i},Z_{i}])+x_{i}\mathbf{B}(X_{i},[Y_{\bar{1}},Z_{\bar{1}}])\\
                                       && +t_{i}\mathbf{B}(X_{\bar{1}},[Y_{i},Z_{\bar{1}}])+t_{i}\mathbf{B}(X_{\bar{1}},[Y_{\bar{1}},Z_{i}])
                                          \Big).
\end{eqnarray*}
Thus
\begin{eqnarray*}
\langle [X,Y]_{\mathfrak{m}},Z\rangle'-\langle X, [Y,Z]_{\mathfrak{m}}\rangle'=
\sum_{i=0}^{s}(x_{i}-t_{i})\Big(
                                    \mathbf{B}([X_{\bar{1}},Y_{\bar{1}}],Z_{i})- \mathbf{B}([X_{i},Y_{\bar{1}}],Z_{\bar{1}})
                               \Big).
\end{eqnarray*}
Setting $t_{i}=x_{i}$ for $0\leq i\leq s$ yields a naturally reductive metric $\langle\cdot,\cdot\rangle'$ on $\mathfrak{m}$. This completes the proof of the proposition.
 \end{proof}

We now compute the Ricci tensor of $(\mathfrak{g},\langle\cdot,\cdot\rangle)$.

\begin{lem}\label{3-lem-levi-civita}
Let $\nabla$ be the Levi-Civita connection of $(\mathfrak{g},\langle\cdot,\cdot\rangle)$, then
\begin{eqnarray*}
\nabla_{X}Y=
\begin{cases}
\frac{1}{2}[X,Y],            &X\in\mathfrak{k}_{i},Y\in\mathfrak{k}_{j};\\
\frac{1}{2}[X,Y],            &X,Y\in\mathfrak{g}_{\bar{1}};\\
\left(1-\frac{x_{i}}{2}\right)[X,Y],    &X\in\mathfrak{k}_{i},Y\in\mathfrak{g}_{\bar{1}};\\
\frac{x_{i}}{2}[X,Y],        &X\in\mathfrak{g}_{\bar{1}},Y\in\mathfrak{k}_{i}.\\
\end{cases}
\end{eqnarray*}
\end{lem}
\begin{proof}
(\romannumeral1) For  $X\in\mathfrak{k}_{i}$, $Y\in\mathfrak{k}_{j}$, $Z\in\mathfrak{g}$, equation \eqref{Eq1} gives
\begin{eqnarray*}
2\langle\nabla_{X}Y,Z\rangle&=&\langle[X,Y],Z\rangle-x_{i}\mathbf{B}(X,[Y,Z])-x_{j}\mathbf{B}(Y,[X,Z])\\
                            &=&\langle[X,Y],Z\rangle+(x_{j}-x_{i})\mathbf{B}([X,Y],Z)\\
                            &=&\langle[X,Y],Z\rangle.
\end{eqnarray*}
Hence $\nabla_{X}Y=\frac{1}{2}[X,Y]$.

(\romannumeral2) For  $X,Y\in\mathfrak{g}_{\bar{1}}$, $Z\in\mathfrak{k}_{i}$, equation \eqref{Eq1} becomes
\begin{eqnarray*}
2x_{i}\mathbf{B}(\nabla_{X}Y,Z)&=&x_{i}\mathbf{B}([X,Y],Z)-\mathbf{B}(X,[Y,Z])+\mathbf{B}(Y,[X,Z])\\
                               &=&x_{i}\mathbf{B}([X,Y],Z),
\end{eqnarray*}
so $\nabla_{X}Y=\frac{1}{2}[X,Y]$.

(\romannumeral3) For  $X\in\mathfrak{k}_{i}$, $Y\in\mathfrak{g}_{\bar{1}}$, $Z\in\mathfrak{g}_{\bar{1}}$, equation \eqref{Eq1} becomes
\begin{eqnarray*}
2\mathbf{B}(\nabla_{X}Y,Z)&=&\mathbf{B}([X,Y],Z)-x_{i}\mathbf{B}(X,[Y,Z])-\mathbf{B}(Y,[X,Z])\\
                               &=&(2-x_{i})\mathbf{B}([X,Y],Z),
\end{eqnarray*}
so $\nabla_{X}Y=\left(1-\frac{x_{i}}{2}\right)[X,Y]$.

(\romannumeral4) For  $X\in\mathfrak{g}_{\bar{1}}$, $Y\in\mathfrak{k}_{i}$,
\begin{eqnarray*}
\nabla_{X}Y&=&\nabla_{Y}X+[X,Y]=\left(1-\frac{x_{i}}{2}\right)[Y,X]+[X,Y]\\
           &=&\frac{x_{i}}{2}[X,Y].
\end{eqnarray*}
\end{proof}

To compute the Ricci tensor, we also need
\begin{lem}\label{3-lem-tradX}
 For all $X,Y\in\mathfrak{g}_{\bar{1}}$ and $0\leq i\leq s$, we have
\begin{eqnarray*}
\left\{
\begin{aligned}
&\mathrm{tr}\,(\mathrm{ad}\,([X,Y]|_{\mathfrak{k}_{i}})|_{\mathfrak{g}_{\bar{1}}})=0,\\
&\mathrm{tr}\,(\mathrm{ad}\,X\circ\mathrm{ad}\,Y|_{\mathfrak{k}_{i}})
=\mathbf{B}\left(X,\mathcal{C}_{\mathbf{B}|_{\mathfrak{k}_{i}}}(Y)\right),\\
&\mathrm{tr}\left(Z\mapsto [X,[Y,Z]|_{\mathfrak{k}_{i}}],Z\in\mathfrak{g}_{\bar{1}}\right)
=-\mathbf{B}\left(X,\mathcal{C}_{\mathbf{B}|_{\mathfrak{k}_{i}}}(Y)\right).
\end{aligned}\right.
\end{eqnarray*}
\end{lem}
\begin{proof}
(i) For $i\neq 0$,  since $\mathfrak{k}_{i}$  is simple,
the first identity holds.
For $i=0$, we have
\begin{equation*}
\mathrm{tr}\,(\mathrm{ad}\,([X,Y]|_{\mathfrak{k}_{0}})|_{\mathfrak{g}_{\bar{1}}})
=\mathrm{tr}\,(\mathrm{ad}\,([X,Y])|_{\mathfrak{g}_{\bar{1}}})
= \mathrm{tr}\,\mathrm{ad}\,([X,Y])
=-\mathrm{str}\,[\mathrm{ad}\,X,\mathrm{ad}\,Y]
=0.
\end{equation*}

(ii) Let $\{e_{j}\}\subset\mathfrak{k}_{i}$ be an arbitrary basis and $\{e_{j}^{\ast}\}\subset\mathfrak{k}_{i}$ its $\mathbf{B}$-dual basis, i.e., $\mathbf{B}(e_{j},e_{k}^{\ast})=\delta_{jk}$. Then
\begin{eqnarray*}
\mathrm{tr}\,(\mathrm{ad}\,X\circ\mathrm{ad}\,Y|_{\mathfrak{k}_{i}})
&=&\sum\limits_{j}\mathbf{B}([X,[Y,e_{j}]],e_{j}^{\ast})\\
&=&\sum\limits_{j}\mathbf{B}(X,[[Y,e_{j}],e_{j}^{\ast}])\\
&=&\sum\limits_{j}\mathbf{B}(X,[e_{j}^{\ast},[e_{j},Y]])\\
 &=&\mathbf{B}(X,\mathcal{C}_{\mathbf{B}|_{\mathfrak{k}_{i}}}(Y)).
\end{eqnarray*}

(iii) Let $\{f_{j}\}\subset\mathfrak{g}_{\bar{1}}$ be a  basis and $\{f_{j}^{\ast}\}\subset\mathfrak{g}_{\bar{1}}$ its $\mathbf{B}$-dual basis, i.e., $\mathbf{B}(f_{j},f_{k}^{\ast})=\delta_{jk}$. Then
\begin{eqnarray*}
\mathrm{tr}\left(Z\mapsto [X,[Y,Z]|_{\mathfrak{k}_{i}}],Z\in\mathfrak{g}_{\bar{1}}\right)
&=&\sum_{j}\mathbf{B}\left(\left[X,[Y,f_{j}]|_{\mathfrak{k}_{i}}\right],f_{j}^{\ast}\right)\\
&=&\sum_{j}\mathbf{B}\left(\left[X,\sum_{k}\mathbf{B}([Y,f_{j}],e_{k}^{\ast})e_{k}\right],f_{j}^{\ast}\right)\\
&=&\sum_{j,k}\mathbf{B}([Y,f_{j}],e_{k}^{\ast})\mathbf{B}([X,e_{k}],f_{j}^{\ast})\\
&=&\sum_{j,k}\mathbf{B}(f_{j},[Y,e_{k}^{\ast}])\mathbf{B}([X,e_{k}],f_{j}^{\ast})\\
&=&\sum_{k}\mathbf{B}([X,e_{k}],[Y,e_{k}^{\ast}])\\
&=&-\sum_{k}\mathbf{B}(X,[e_{k},[e_{k}^{\ast},Y]])\\
&=&-\mathbf{B}\left(X,\mathcal{C}_{\mathbf{B}|_{\mathfrak{k}_{i}}}(Y)\right).
\end{eqnarray*}
\end{proof}
\begin{cor}\label{3-cor-B=K}
With the above notation,
$$\mathbf{K}(X,Y)=2\sum_{i=0}^{s}\mathbf{B}\left(X,\mathcal{C}_{\mathbf{B}|_{\mathfrak{k}_{i}}}(Y)\right),
\quad \forall X,Y\in \mathfrak{g}_{\bar{1}}.
$$
\end{cor}
\begin{proof}
It follows from Lemma \ref{3-lem-tradX} that
\begin{eqnarray*}
  \mathbf{K}(X,Y) &=& \mathrm{str}\,(\mathrm{ad}\,X\circ\mathrm{ad}\,Y) \\
   &=&  \mathrm{tr}\,(\mathrm{ad}\,X\circ\mathrm{ad}\,Y|_{\mathfrak{g}_{\bar{0}}})
   -\mathrm{tr}\,(\mathrm{ad}\,X\circ\mathrm{ad}\,Y|_{\mathfrak{g}_{\bar{1}}}) \\
   &=& \sum_{i=0}^{s}\Big(\mathrm{tr}\,(\mathrm{ad}\,X\circ\mathrm{ad}\,Y|_{\mathfrak{k}_{i}})
   -\mathrm{tr}\left(Z\mapsto [X,[Y,Z]|_{\mathfrak{k}_{i}}],Z\in\mathfrak{g}_{\bar{1}}\right)\Big) \\
   &=& 2\sum_{i=0}^{s}\mathbf{B}\left(X,\mathcal{C}_{\mathbf{B}|_{\mathfrak{k}_{i}}}(Y)\right).
\end{eqnarray*}
\end{proof}
\begin{prop}\label{3-prop-ric-tensor}
The Ricci tensor $\mathrm{ric}$ of $(\mathfrak{g},\langle\cdot,\cdot\rangle)$ is given by
\begin{eqnarray*}
\mathrm{ric}(X,Y)=
\begin{cases}
-\frac{x_{0}^{2}}{4}\mathbf{K}(X,Y),       &X,Y\in\mathfrak{k}_{0};\\
\frac{1}{4}(l_{i}x_{i}^{2}-1)\mathbf{K}_{i}(X,Y),                                               &X,Y\in\mathfrak{k}_{i}, 1\leq i\leq s;\\
\sum\limits_{i=0}^{s}\left(\frac{x_{i}}{2}-1\right)\mathbf{B}(X,\mathcal{C}_{\mathbf{B}|_{\mathfrak{k}_{i}}}(Y)),    &X,Y\in\mathfrak{g}_{\bar{1}};\\
0,                                                                                              &\mathrm{otherwise}.\\
\end{cases}
\end{eqnarray*}
\end{prop}
\begin{proof}
(\romannumeral1) Let $X\in\mathfrak{k}_{i}$, $Y\in\mathfrak{k}_{j}$. Then for  $Z\in\mathfrak{g}_{\bar{0}}$,
\begin{eqnarray*}
R(Z,X)Y&=&\nabla_{Z}\nabla_{X}Y-\nabla_{X}\nabla_{Z}Y-\nabla_{[Z,X]}Y\\
       &=&\frac{1}{4}[Z,[X,Y]]-\frac{1}{4}[X,[Z,Y]]-\frac{1}{2}[[Z,X],Y]\\
       &=&-\frac{1}{4}[Y,[X,Z]].
\end{eqnarray*}
For $Z\in\mathfrak{g}_{\bar{1}}$,
\begin{eqnarray*}
R(Z,X)Y&=&\nabla_{Z}\nabla_{X}Y-\nabla_{X}\nabla_{Z}Y-\nabla_{[Z,X]}Y\\
       &=&\frac{1}{2}\nabla_{Z}[X,Y]-\frac{x_{j}}{2}\nabla_{X}[Z,Y]-\frac{x_{j}}{2}[[Z,X],Y]\\
       &=&\frac{x_{j}}{4}[Z,[X,Y]]-\frac{x_{j}}{2}\left(1-\frac{x_{i}}{2}\right)[X,[Z,Y]]-\frac{x_{j}}{2}[[Z,X],Y]\\
       &=&-\frac{x_{j}}{4}[Y,[X,Z]]+\frac{x_{j}(1-x_{i})}{4}[X,[Y,Z]].
\end{eqnarray*}
Then
\begin{eqnarray*}
\mathrm{ric}(X,Y)&=&\mathrm{str}(Z\mapsto R(Z,X)Y)\\
                 &=&-\frac{1}{4}\mathbf{K}_{\bar{0}}(Y,X)
                 +\frac{x_{j}}{4}\mathrm{tr}\,(\mathrm{ad}\,Y\,\mathrm{ad}\,X|_{\mathfrak{g}_{\bar{1}}})
                    -\frac{x_{j}(1-x_{i})}{4}\mathrm{tr}\,(\mathrm{ad}\,X\,\mathrm{ad}\,Y|_{\mathfrak{g}_{\bar{1}}})\\
                 &=&-\frac{1}{4}\mathbf{K}_{\bar{0}}(Y,X)+\frac{x_{j}}{4}(\mathbf{K}_{\bar{0}}(Y,X)-\mathbf{K}(Y,X)) -\frac{x_{j}(1-x_{i})}{4}(\mathbf{K}_{\bar{0}}(X,Y)-\mathbf{K}(X,Y))\\
                 &=&\frac{1}{4}(x_{i}x_{j}-1)\mathbf{K}_{\bar{0}}(X,Y)-\frac{x_{i}x_{j}}{4}\mathbf{K}(X,Y),
\end{eqnarray*}
where $\mathbf{K}_{\bar{0}}$ is the Killing form of $\mathfrak{g}_{\bar{0}}$.
This implies that
\begin{eqnarray*}
&&\mathrm{ric}(\mathfrak{k}_{i},\mathfrak{k}_{j})=0, i\neq j,\\
&&\mathrm{ric}(X,Y)=-\frac{x_{0}^{2}}{4}\mathbf{K}(X,Y), \, \forall X,Y\in\mathfrak{k}_{0}.
\end{eqnarray*}
Since $\mathbf{K}|_{\mathfrak{k}_{i}}=(1-l_{i})\mathbf{K}_{i}$, we get
\begin{eqnarray*}
 \mathrm{ric}(X,Y)&=&\frac{1}{4}(x_{i}^{2}-1)\mathbf{K}_{i}(X,Y)-\frac{x_{i}^{2}}{4}(1-l_{i})\mathbf{K}_{i}(X,Y)\\
                  &=&\frac{1}{4}(l_{i}x_{i}^{2}-1)\mathbf{K}_{i}(X,Y), \,\forall X,Y\in\mathfrak{k}_{i}.
\end{eqnarray*}

(\romannumeral2) Let $X,Y\in\mathfrak{g}_{\bar{1}}$. Then for  $Z\in \mathfrak{k}_{i}$,
\begin{eqnarray*}
R(Z,X)Y&=&\nabla_{Z}\nabla_{X}Y-\nabla_{X}\nabla_{Z}Y-\nabla_{[Z,X]}Y\\
       &=&\frac{1}{2}\nabla_{Z}[X,Y]-\left(1-\frac{x_{i}}{2}\right)\nabla_{X}[Z,Y]-\frac{1}{2}[[Z,X],Y]\\
       &=&\frac{1}{4}[Z,[X,Y]]-\frac{1}{2}\left(1-\frac{x_{i}}{2}\right)[X,[Z,Y]]-\frac{1}{2}[[Z,X],Y]\\
       &=&-\frac{1}{4}(1-x_{i})[X,[Z,Y]]-\frac{1}{4}[[Z,X],Y]\\
       &=&\frac{1}{4}(1-x_{i})[X,[Y,Z]]+\frac{1}{4}[Y,[X,Z]].
\end{eqnarray*}

For  $Z\in\mathfrak{g}_{\bar{1}}$,
\begin{eqnarray*}
R(Z,X)Y&=&\nabla_{Z}\nabla_{X}Y+\nabla_{X}\nabla_{Z}Y-\nabla_{[Z,X]}Y\\
       &=&\frac{1}{2}\nabla_{Z}[X,Y]+\frac{1}{2}\nabla_{X}[Z,Y]-\nabla_{[Z,X]}Y\\
       &=&\sum_{i=0}^{s}\left(\frac{x_{i}}{4}[Z,[X,Y]|_{\mathfrak{k}_{i}}]+\frac{x_{i}}{4}[X,[Z,Y]|_{\mathfrak{k}_{i}}]-\left(1-\frac{x_{i}}{2}\right)[[Z,X]|_{\mathfrak{k}_{i}},Y]\right)\\
       &=&\sum_{i=0}^{s}\left(-\frac{x_{i}}{4}[[X,Y]|_{\mathfrak{k}_{i}},Z]+\frac{x_{i}}{4}[X,[Y,Z]|_{\mathfrak{k}_{i}}]+\left(1-\frac{x_{i}}{2}\right)[Y,[X,Z]|_{\mathfrak{k}_{i}}]\right).
\end{eqnarray*}
By Lemma \ref{3-lem-tradX}, we obtain
\begin{eqnarray*}
\mathrm{ric}(X,Y)&=&\mathrm{str}(Z\mapsto R(Z,X)Y)\\
                 &=&\sum\limits_{i=0}^{s}\frac{1}{4}(1-x_{i})\mathrm{tr}\,(\mathrm{ad}\,X\circ\mathrm{ad}\,Y|_{\mathfrak{k}_{i}})
                    +\frac{1}{4}\mathrm{tr}\,(\mathrm{ad}\,Y\circ\mathrm{ad}\,X|_{\mathfrak{g}_{\bar{0}}})\\
                 &&-\sum_{i=0}^{s}\left(\left(-\frac{x_{i}}{4}\right)\mathbf{B}\left(X,\mathcal{C}_{\mathbf{B}|_{\mathfrak{k}_{i}}}(Y)\right)
                    -\left(1-\frac{x_{i}}{2}\right)\mathbf{B}\left(Y,\mathcal{C}_{\mathbf{B}|_{\mathfrak{k}_{i}}}(X)\right) \right)\\
                 &=&\sum_{i=0}^{s}\left(\frac{x_{i}}{2}-1\right)\mathbf{B}\left(X,\mathcal{C}_{\mathbf{B}|_{\mathfrak{k}_{i}}}(Y)\right).
\end{eqnarray*}
This completes the proof of the proposition.
\end{proof}
\begin{thm}\label{3-thm-main}
$(\mathfrak{g},\langle\cdot,\cdot\rangle)$ is Einstein with $\mathrm{ric}=c\langle\cdot,\cdot\rangle$ if and only if the following equations hold:
\begin{eqnarray}\label{Einstein equiivalent}
\left\{
\begin{aligned}
&-\frac{x_{0}}{4}\mathbf{K}|_{\mathfrak{k}_{0}}=c\mathbf{B}|_{\mathfrak{k}_{0}},  \\
&\frac{1}{4}(l_{i}x_{i}^{2}-1)=c x_{i}b_{i},            \quad 1\leq i\leq s,\\
&\sum\limits_{i=0}^{s}\left(\frac{x_{i}}{2}-1\right)\mathcal{C}_{\mathbf{B}|_{\mathfrak{k}_{i}}}=c\, \mathrm{Id}|_{\mathfrak{g}_{\bar{1}}}.
\end{aligned}\right.
\end{eqnarray}
\end{thm}
\begin{rem}
Equation \eqref{Einstein equiivalent}  involves only $\mathbf{K}$ and $\mathbf{B}$. If we extend $\mathbf{B}$ naturally to a bi-invariant metric on the complex Lie superalgebra $\mathfrak{g}\otimes \mathbb{C}$, then \eqref{Einstein equiivalent} can be viewed as the Einstein equation for $\mathfrak{g}\otimes \mathbb{C}$.
If $\mathfrak{g}$  is a complex Lie superalgebra regarded as  real, then
 \eqref{Einstein equiivalent} is also  the Einstein equation for the complex  Lie superalgebra $\mathfrak{g}$.
\end{rem}
\begin{rem}
 Let $\Lambda$  denote the set comprising $\{0\}$, basic classical Lie
superalgebras and the one-dimensional Lie algebra. In \cite{ABB09,Benayadi00}, it is shown that
a Lie superalgebra $\mathfrak{g}=\mathfrak{g}_{\bar{0}}+\mathfrak{g}_{\bar{1}}$ with $\mathfrak{g}_{\bar{0}}$ reductive  admitting a non-degenerate even supersymmetric bilinear form $\mathbf{B}$ is either an
element of $\Lambda$ or obtained from a finite number of elements of $\Lambda$ by a finite sequence of double extensions by the one-dimensional Lie algebra, and/or generalized double extensions by the one-dimensional Lie superalgebra, and/or by orthogonal direct sums.
This implies that, except for basic classical Lie superalgebras and their direct sums, every quadratic Lie superalgebra $(\mathfrak{g},\mathbf{B})$ with $\mathfrak{g}_{\bar{0}}$ reductive has a non-zero homogeneous center,
i.e., $C(\mathfrak{g})\cap (\mathfrak{g}_{\bar{0}}\cup \mathfrak{g}_{\bar{1}})\neq 0$.
From  \eqref{Einstein equiivalent},  it follows that if $\mathfrak{g}$ is not an orthogonal direct
sum of  basic classical Lie superalgebras, then any Einstein metric $\langle\cdot,\cdot\rangle$ of the form \eqref{Einstein-metric-form} must be Ricci flat.
\end{rem}

\section{Classification}\label{sect4}
In this section, we study the Einstein equation \eqref{Einstein equiivalent} for  real basic classical Lie superalgebras. Note that a non-degenerate  even supersymmetric bi-invariant bilinear form on a real basic classical Lie superalgebra is unique up to a scaling.
\begin{lem}\label{4-lem-Casimir}
Let $\mathfrak{g}=\mathfrak{g}_{\bar{0}}+\mathfrak{g}_{\bar{1}}$ be a real basic classical Lie superalgebra.
 Then
\begin{eqnarray*}
\mathcal{C}_{\mathbf{B}|_{\mathfrak{k}_{i}}}=\frac{l_{i}\mathrm{dim}\,\mathfrak{k}_{i}}{b_{i}\mathrm{dim}\,\mathfrak{g}_{\bar{1}}}\mathrm{Id}|_{\mathfrak{g}_{\bar{1}}}, \, 1\leq i\leq s.
\end{eqnarray*}
Moreover, if the Killing form $\mathbf{K}$ of $\mathfrak{g}$ is non-degenerate and $\mathfrak{k}_{0}\neq 0$, then
\begin{eqnarray*}
\mathcal{C}_{\mathbf{K}|_{\mathfrak{k}_{0}}}=-\frac{\mathrm{dim}\,\mathfrak{k}_{0}}{\mathrm{dim}\,\mathfrak{g}_{\bar{1}}}\mathrm{Id}|_{\mathfrak{g}_{\bar{1}}}.
\end{eqnarray*}
\end{lem}
\begin{proof}
By \cite[Proposition 2.1.2 and Proposition 5.3.1]{Kac},  the irreducible representation spaces of $\mathfrak{k}_{i}$ on $\mathfrak{g}_{\bar{1}}$ are either isomorphic or contragredient, so $\mathcal{C}_{\mathbf{B}|_{\mathfrak{k}_{i}}}=\lambda\mathrm{Id}|_{\mathfrak{g}_{\bar{1}}}$ for some $\lambda\in\mathbb{R}$.
This implies that
\begin{eqnarray*}
\lambda \mathrm{dim}\,\mathfrak{g}_{\bar{1}}=\mathrm{tr}\,\mathcal{C}_{\mathbf{B}|_{\mathfrak{k}_{i}}}
=\frac{l_{i}\mathrm{dim}\,\mathfrak{k}_{i}}{b_{i}},
\end{eqnarray*}
and thus $\lambda=\frac{l_{i}\mathrm{dim}\,\mathfrak{k}_{i}}{b_{i}\mathrm{dim}\,\mathfrak{g}_{\bar{1}}}$.

Moreover, if  $\mathbf{K}$ is non-degenerate and $\mathfrak{k}_{0}\neq 0$, then $\mathfrak{g}$ (or $\mathfrak{g}\otimes \mathbb{C}$) is isomorphic to $\mathbf{A}(m,n)$ with $m\neq n$ or $\mathbf{C}(n)$.
 The representation of $\mathfrak{k}_{0}$ on $\mathfrak{g}_{\bar{1}}$ is a multiple of identity map,
 so $\mathrm{tr}\,\mathcal{C}_{\mathbf{K}|_{\mathfrak{k}_{0}}}=-\mathrm{dim}\,\mathfrak{k}_{0}$,
 and hence $\mathcal{C}_{\mathbf{K}|_{\mathfrak{k}_{0}}}
=-\frac{\mathrm{dim}\,\mathfrak{k}_{0}}{\mathrm{dim}\,\mathfrak{g}_{\bar{1}}}\mathrm{Id}|_{\mathfrak{g}_{\bar{1}}}$.
\end{proof}
\begin{cor}\label{4-cor-K-B}
 Let $\mathfrak{g}=\mathfrak{g}_{\bar{0}}+\mathfrak{g}_{\bar{1}}$ be a real basic classical Lie superalgebra.

 (i) If $\mathbf{K}$ is non-degenerate, then
$$-\frac{\mathrm{dim}\,\mathfrak{k}_{0}}{\mathrm{dim}\,\mathfrak{g}_{\bar{1}}}
  +\sum_{i=1}^{s}\frac{l_{i}\mathrm{dim}\,\mathfrak{k}_{i}}{(1-l_{i})\mathrm{dim}\,\mathfrak{g}_{\bar{1}}}
  =\frac{1}{2}.$$

 (ii) If $\mathbf{K}\equiv0$, then
 $$ \sum_{i=1}^{s}\frac{l_{i}\mathrm{dim}\,\mathfrak{k}_{i}}{(1-l_{i})\mathrm{dim}\,\mathfrak{g}_{\bar{1}}}
  =0.$$
\end{cor}
\begin{proof}
 This follows from Corollary \ref{3-cor-B=K}, Lemma \ref{4-lem-Casimir}, and the fact that $\mathfrak{k}_{0}=0$ if $\mathbf{K}\equiv0$.
\end{proof}
Combining Theorem \ref{3-thm-main}, Lemma \ref{4-lem-Casimir} and Corollary \ref{4-cor-K-B}, we obtain
\begin{thm}\label{4-thm-B=K}
 Suppose $\mathbf{K}$ is non-degenerate and let $\mathbf{B}=\mathbf{K}$.
  Then  equation \eqref{Einstein equiivalent} becomes
\begin{eqnarray}\label{B=K equivalent}
\left\{
\begin{aligned}
&c=-\frac{x_{0}}{4},\quad \mathrm{if}\, \mathfrak{k}_{0}\neq0, \\
&\frac{1}{4}(l_{i}x_{i}^{2}-1)=c (1-l_{i})x_{i},\quad 1\leq i\leq s,\\
&-\frac{\mathrm{dim}\,\mathfrak{k}_{0}}{\mathrm{dim}\,\mathfrak{g}_{\bar{1}}}x_{0}
+\sum\limits_{i=1}^{s}\frac{l_{i}\mathrm{dim}\,\mathfrak{k}_{i}}{(1-l_{i})\mathrm{dim}\,\mathfrak{g}_{\bar{1}}}x_{i}=2c+1.
\end{aligned}\right.
\end{eqnarray}
\end{thm}

Combining Proposition \ref{2-prop-index}, Theorem \ref{3-thm-main} and Lemma \ref{4-lem-Casimir},
we obtain the following reduction principle, which allows us to relate the Einstein equation for a real basic classical Lie superalgebra to that for its complexification.
\begin{prop}\label{4-prop-solution-real-complex}
Let $\mathfrak{g}\otimes\mathbb{C}$ be a complex basic classical Lie superalgebra, and let $\mathfrak{g}=\mathfrak{g}_{\bar{0}}+\mathfrak{g}_{\bar{1}}$ be a real form of
$\mathfrak{g}\otimes\mathbb{C}$ with decomposition  $\mathfrak{g}_{\bar{0}}=\mathfrak{k}_{0}\oplus\mathfrak{k}_{1}\oplus\cdots\oplus\mathfrak{k}_{s}$ as above. Suppose
$\mathfrak{k}_{i}\otimes\mathbb{C}$ is simple for  $1\leq i\leq r$,
and $\mathfrak{k}_{j}\otimes\mathbb{C}=\mathfrak{k}'_{j}\oplus\mathfrak{k}''_{j}$ is a direct sum of two ideals
for $r+1\leq j\leq s$. Then, for a given bi-invariant metric $\mathbf{B}$ on $\mathfrak{g}$
 (naturally extended to $\mathfrak{g}\otimes\mathbb{C}$),
$(x_{0},x_{1},\ldots,x_{s})$ is a solution of equation \eqref{Einstein equiivalent} for $\mathfrak{g}$ if and only if
$(x_{0},x_{1},\ldots,x_{r},x_{r+1},x_{r+1},\ldots,x_{s},x_{s})$ is a solution of equation \eqref{Einstein equiivalent} for $\mathfrak{g}\otimes\mathbb{C}$.
\end{prop}

Finally, in this work, we classify all real solutions of equation \eqref{Einstein equiivalent} for each complex basic classical Lie superalgebra  using Maple. By virtue of Proposition 4.4, the corresponding real solutions for real basic classical Lie superalgebras can be directly obtained,
thereby completing the proof of Theorem \ref{1-thm-main}.

1. Case $\mathbf{A}(m,n)$:
\begin{eqnarray*}
\mathfrak{g}=\mathbf{A}(m,n)=\mathfrak{sl}(m+1,n+1; \mathbb{C})=
\left\{\left(\begin{array}{cc}X&Y\\
                          Z&W
\end{array}\right)\bigg{|}\mathrm{tr}\,X-\mathrm{tr}\,W=0\right\}, m\neq n,m,n\geq 0.
\end{eqnarray*}
$$\mathfrak{g}_{\bar{0}}=\left\{\left(\begin{array}{cc}X&0\\
                          0&W
\end{array}\right)\bigg{|}\mathrm{tr}\,X-\mathrm{tr}\,W=0\right\}
=\mathbb{C}\oplus \mathbf{A}_{m}\oplus \mathbf{A}_{n},
\quad \mathfrak{g}_{\bar{1}}=\left\{\left(\begin{array}{cc}0&Y\\
                          Z&0
\end{array}\right)\right\},$$
 the Lie superbracket is  $[X,Y]=XY-(-1)^{[X][Y]}YX$, $\forall X,Y\in \mathbf{A}(m,n)$.

  We have
  $\mathrm{dim}\,\mathfrak{k}_{0}=1$, $\mathrm{dim}\,\mathfrak{k}_{1}=(m+1)^{2}-1$, $\mathrm{dim}\,\mathfrak{k}_{2}=(n+1)^{2}-1$, $\mathrm{dim}\,\mathfrak{g}_{\bar{1}}=2(m+1)(n+1)$, $l_{1}=\frac{n+1}{m+1}$, $l_{2}=\frac{m+1}{n+1}$.
Then
\begin{eqnarray*}
\left\{
\begin{aligned}
&\frac{\mathrm{dim}\,\mathfrak{k}_{0}}{\mathrm{dim}\,\mathfrak{g}_{\bar{1}}}=\frac{1}{2(m+1)(n+1)},\\
&\frac{l_{1}\mathrm{dim}\,\mathfrak{k}_{1}}{(1-l_{1})\mathrm{dim}\,\mathfrak{g}_{\bar{1}}}=\frac{m(m+2)}{2(m-n)(m+1)},\\
&\frac{l_{2}\mathrm{dim}\,\mathfrak{k}_{2}}{(1-l_{2})\mathrm{dim}\,\mathfrak{g}_{\bar{1}}}=\frac{n(n+2)}{2(n-m)(n+1)}.
\end{aligned}
\right.
\end{eqnarray*}

Let $\mathbf{B}=\mathbf{K}$ and then equation \eqref{B=K equivalent} becomes
\begin{eqnarray*}
\left\{
\begin{aligned}
&c=-\frac{x_{0}}{4},\\
&\frac{1}{4}\left(\frac{n+1}{m+1}x_{1}^{2}-1\right)=c\left(1-\frac{n+1}{m+1}\right)x_{1},\\
&\frac{1}{4}\left(\frac{m+1}{n+1}x_{2}^{2}-1\right)=c\left(1-\frac{m+1}{n+1}\right)x_{2},\\
&-\frac{1}{2(m+1)(n+1)}\cdot x_{0}+\frac{m(m+2)}{2(m-n)(m+1)}\cdot x_{1}+\frac{n(n+2)}{2(n-m)(n+1)}\cdot x_{2}=2c+1.
\end{aligned}\right.
\end{eqnarray*}
There is a unique real solution $(x_{0},x_{1},x_{2},c)=\left(1,1,1,-\frac{1}{4}\right)$.

2. Case $\mathbf{A}(n,n)$:
\begin{eqnarray*}
\mathfrak{g}=\mathbf{A}(n,n)=\mathfrak{sl}(n+1,n+1; \mathbb{C})/\langle I_{2n+2}\rangle,n\geq 0.
\end{eqnarray*}
We have
$\mathfrak{g}_{\bar{0}}=\mathbf{A}_{n}\oplus \mathbf{A}_{n}$, $\mathrm{dim}\,\mathfrak{k}_{1}=\mathrm{dim}\,\mathfrak{k}_{2}=(n+1)^{2}-1$, $\mathrm{dim}\,\mathfrak{g}_{\bar{1}}=2(n+1)^{2}$.
 Define $\mathbf{B}$  on $\mathbf{A}(n,n)$  by $\mathbf{B}(\widetilde{X},\widetilde{Y})=2(n+1)\mathrm{str}\,(XY)$
for  $X,Y\in \mathfrak{sl}(n+1,n+1; \mathbb{C})$. Then $\mathbf{B}|_{\mathfrak{k}_{1}}=\mathbf{K}_{1}$ and  $\mathbf{B}|_{\mathfrak{k}_{2}}=-\mathbf{K}_{2}$, so $b_{1}=1$, $b_{2}=-1$. Since  $l_{1}=l_{2}=1$, we have
\begin{eqnarray*}
&&\mathcal{C}_{\mathbf{B}|_{\mathfrak{k}_{1}}}=\frac{\mathrm{dim}\,\mathfrak{k}_{1}}{\mathrm{dim}\,\mathfrak{g}_{\bar{1}}}\mathrm{Id}|_{\mathfrak{g}_{\bar{1}}}
                                            =\frac{n(n+2)}{2(n+1)^{2}}\mathrm{Id}|_{\mathfrak{g}_{\bar{1}}},\\
&&\mathcal{C}_{\mathbf{B}|_{\mathfrak{k}_{2}}}=-\frac{\mathrm{dim}\,\mathfrak{k}_{2}}{\mathrm{dim}\,\mathfrak{g}_{\bar{1}}}\mathrm{Id}|_{\mathfrak{g}_{\bar{1}}}
                                              =-\frac{n(n+2)}{2(n+1)^{2}}\mathrm{Id}|_{\mathfrak{g}_{\bar{1}}}.
\end{eqnarray*}
Equation (\ref{Einstein equiivalent}) becomes
\begin{eqnarray*}
\left\{
\begin{aligned}
&\frac{1}{4}(x_{1}^{2}-1)=c x_{1},\\
&\frac{1}{4}(x_{2}^{2}-1)=-c x_{2},\\
&\frac{n(n+2)}{4(n+1)^{2}}(x_{1}-x_{2})=c.
\end{aligned}\right.
\end{eqnarray*}
There are two real solutions $(x_{1},x_{2},c)=(1,1,0),(-1,-1,0)$.

3. Case $\mathbf{B}(m,n)$:
\begin{eqnarray*}
\mathfrak{g}=\mathbf{B}(m,n)=\mathfrak{osp}(2m+1,2n)=\mathfrak{osp}(2m+1,2n)_{\bar{0}}+\mathfrak{osp}(2m+1,2n)_{\bar{1}},m\geq 0,n> 0,
\end{eqnarray*}
where
\begin{eqnarray*}
\mathfrak{osp}(l,k)_{s}=\{A\in\mathfrak{gl}(l,k;\mathbb{C})_{s}|F(A(X),Y)=-(-1)^{s\cdot[X]}F(X,A(Y))\}, s\in \mathbb{Z}_{2},
\end{eqnarray*}
and $F$ is a non-degenerate even supersymmetric bilinear form on $\mathbb{C}^{l|k}$.

We have $\mathfrak{g}_{\bar{0}}=\mathbf{B}_{m}\oplus \mathbf{C}_{n}$, $\mathrm{dim}\,\mathfrak{k}_{1}=(2m+1)m$, $\mathrm{dim}\,\mathfrak{k}_{2}=(2n+1)n$, $\mathrm{dim}\,\mathfrak{g}_{\bar{1}}=2n(2m+1)$, $l_{1}=\frac{2n}{2m-1}$, $l_{2}=\frac{2m+1}{2n+2}$. Then
\begin{eqnarray*}
\frac{l_{1}\mathrm{dim}\,\mathfrak{k}_{1}}{(1-l_{1})\mathrm{dim}\,\mathfrak{g}_{\bar{1}}}=\frac{m}{2m-2n-1},\quad
\frac{l_{2}\mathrm{dim}\,\mathfrak{k}_{2}}{(1-l_{2})\mathrm{dim}\,\mathfrak{g}_{\bar{1}}}=-\frac{2n+1}{2(2m-2n-1)}.
\end{eqnarray*}
Let $\mathbf{B}=\mathbf{K}$ and  equation \eqref{B=K equivalent} becomes
\begin{eqnarray*}
\left\{
\begin{aligned}
&\frac{1}{4}\left(\frac{2n}{2m-1}x_{1}^{2}-1\right)=c\left(1-\frac{2n}{2m-1}\right)x_{1},\\
&\frac{1}{4}\left(\frac{2m+1}{2n+2}x_{2}^{2}-1\right)=c\left(1-\frac{2m+1}{2n+2}\right)x_{2},\\
&\frac{m}{2m-2n-1}x_{1}-\frac{2n+1}{2(2m-2n-1)}x_{2}=2c+1.
\end{aligned}\right.
\end{eqnarray*}
Solving this system yields
\begin{eqnarray*}
\left\{
\begin{aligned}
&c=\frac{2nx_{1}^{2}-2m+1}{4(2m-2n-1)x_{1}},\\
&x_{2}=\frac{2(m-n)x_{1}^{2}-2(2m-2n-1)x_{1}+2m-1}{(2n+1)x_{1}},
\end{aligned}\right.
\end{eqnarray*}
where $x_{1}$ satisfies  the quartic equation
\begin{eqnarray*}
(x_{1}-1)(A x_{1}^{3}+B x_{1}^{2}+C x_{1}+D)=0,
\end{eqnarray*}
with coefficients
\begin{eqnarray*}
\left\{
\begin{aligned}
&A=2\left( 2m^{3}+(-4n+1)m^{2}+n(4n-1)m-2n^{3}\right),\\
&B=-2\left( 6m^{3}-(12n+1)m^{2}+(8n^{2}-n-2)m-2n^{3}+n\right),\\
&C=(2m-1)\left( 6m^{2}-(10n+1)m+4n^{2}-n-1\right),\\
&D=(2m-1)^{2}(n-m).
\end{aligned}\right.
\end{eqnarray*}
Since $A+B+C+D=0$,
the equation admits at least two distinct real solutions.

4. Case $\mathbf{C}(n)$:
\begin{eqnarray*}
\mathfrak{g}=\mathbf{C}(n)=\mathfrak{osp}(2,2n-2),n\geq 3.
\end{eqnarray*}
We have
$\mathfrak{g}_{\bar{0}}=\mathbb{C}\oplus \mathbf{C}_{n-1}$, $\mathrm{dim}\,\mathfrak{k}_{0}=1$, $\mathrm{dim}\,\mathfrak{k}_{1}=(n-1)(2n-1)$, $\mathrm{dim}\,\mathfrak{g}_{\bar{1}}=4(n-1)$, $l_{1}=\frac{1}{n}$. Then
\begin{eqnarray*}
\frac{\mathrm{dim}\,\mathfrak{k}_{0}}{\mathrm{dim}\,\mathfrak{g}_{\bar{1}}}=\frac{1}{4(n-1)},\quad
\frac{l_{1}\mathrm{dim}\,\mathfrak{k}_{1}}{(1-l_{1})\mathrm{dim}\,\mathfrak{g}_{\bar{1}}}=\frac{2n-1}{4(n-1)}.
\end{eqnarray*}
Let $\mathbf{B}=\mathbf{K}$ and  equation (\ref{B=K equivalent}) becomes
\begin{eqnarray*}
\left\{
\begin{aligned}
&c=-\frac{x_{0}}{4},\\
&\frac{1}{4}\left(\frac{1}{n}x_{1}^{2}-1\right)=c\left(1-\frac{1}{n}\right)x_{1},\\
&-\frac{1}{4(n-1)}x_{0}+\frac{2n-1}{4(n-1)}x_{1}=2c+1.
\end{aligned}\right.
\end{eqnarray*}
There are two real solutions
$(x_{0},x_{1},c)=\left(1,1,-\frac{1}{4}\right)$ and
$$\left(\frac{4n^{3}-20 n^{2}+33n-16}{4n^{3}-16n^{2}+23n-12},
\frac{2n^{2}-3n}{2n^{2}-5n+4},
\frac{-4n^{3}+20 n^{2}-33n+16}{4(4n^{3}-16 n^{2}+23n-12)}\right).$$

5. Case $\mathbf{D}(m,n)$, $m-n\neq 1$:
\begin{eqnarray*}
\mathfrak{g}=\mathbf{D}(m,n)=\mathfrak{osp}(2m,2n),m\geq 2,n>0.
\end{eqnarray*}
We have
$\mathfrak{g}_{\bar{0}}=\mathbf{D}_{m}\oplus \mathbf{C}_{n}$, $\mathrm{dim}\,\mathfrak{k}_{1}=m(2m-1)$, $\mathrm{dim}\,\mathfrak{k}_{2}=n(2n+1)$, $\mathrm{dim}\,\mathfrak{g}_{\bar{1}}=4mn$, $l_{1}=\frac{n}{m-1}$, $l_{2}=\frac{m}{n+1}$. Then
\begin{eqnarray*}
\frac{l_{1}\mathrm{dim}\,\mathfrak{k}_{1}}{(1-l_{1})\mathrm{dim}\,\mathfrak{g}_{\bar{1}}}=\frac{2m-1}{4(m-n-1)},\quad
\frac{l_{2}\mathrm{dim}\,\mathfrak{k}_{2}}{(1-l_{2})\mathrm{dim}\,\mathfrak{g}_{\bar{1}}}=-\frac{2n+1}{4(m-n-1)}.
\end{eqnarray*}
Let $\mathbf{B}=\mathbf{K}$ and equation \eqref{B=K equivalent} becomes
\begin{eqnarray*}
\left\{
\begin{aligned}
&\frac{1}{4}\left(\frac{n}{m-1}x_{1}^{2}-1\right)=c\left(1-\frac{n}{m-1}\right)x_{1},\\
&\frac{1}{4}\left(\frac{m}{n+1}x_{2}^{2}-1\right)=c\left(1-\frac{m}{n+1}\right)x_{2},\\
&\frac{2m-1}{4(m-n-1)}x_{1}-\frac{2n+1}{4(m-n-1)}x_{2}=2c+1.
\end{aligned}\right.
\end{eqnarray*}
Solving this system yields
\begin{eqnarray*}
\left\{
\begin{aligned}
&c=\frac{nx_{1}^{2}-m+1}{4(m-n-1)x_{1}},\\
&x_{2}=\frac{(2m-2n-1)x_{1}^{2}-4(m-n-1)x_{1}+2(m-1)}{(2n+1)x_{1}},
\end{aligned}\right.
\end{eqnarray*}
where  $x_{1}$ satisfies the quartic equation
\begin{eqnarray*}
(x_{1}-1)(A x_{1}^{3}+B x_{1}^{2}+C x_{1}+D)=0,
\end{eqnarray*}
with coefficients
\begin{eqnarray*}
\left\{
\begin{aligned}
&A=4m^{3}-4(2n+1)m^{2}+( 8n^{2}+6n+1)m-n(2n+1)^{2},\\
&B=-12m^{3}+4(6n+5)m^{2}-(16n^{2}+22n+7)m+n(4n^{2}+8n+3),\\
&C=2(m-1)\left( 6m^{2}-m(10n+7)+(2n+1)^{2}\right),\\
&D=2(m-1)^{2}(2n-2m+1).
\end{aligned}\right.
\end{eqnarray*}
The condition $m\neq n$ yields $A+B+C+D\neq0$,
so there are at least two   distinct real solutions of equation \eqref{B=K equivalent}.

6. Case $\mathbf{D}(n+1,n)$, $n\neq 1.$
We have
$\mathfrak{g}_{\bar{0}}=\mathfrak{so}(2n+2)\oplus \mathfrak{sp}(2n)$, $\mathrm{dim}\,\mathfrak{k}_{1}=(n+1)(2n+1)$, $\mathrm{dim}\,\mathfrak{k}_{2}=n(2n+1)$, $\mathrm{dim}\,\mathfrak{g}_{\bar{1}}=4n(n+1)$, $l_{1}=l_{2}=1$.
 Define $\mathbf{B}$  on $\mathbf{D}(n+1,n)$  by $\mathbf{B}(X,Y)=2n\,\mathrm{str}\,(XY)$. Then $\mathbf{B}|_{\mathfrak{k}_{1}}=\mathbf{K}_{1}$,  $\mathbf{B}|_{\mathfrak{k}_{2}}=-\frac{n}{n+1}\mathbf{K}_{2}$, so $b_{1}=1$, $b_{2}=-\frac{n}{n+1}$. We have
\begin{eqnarray*}
&&\mathcal{C}_{\mathbf{B}|_{\mathfrak{k}_{1}}}=\frac{\mathrm{dim}\,\mathfrak{k}_{1}}{\mathrm{dim}\,\mathfrak{g}_{\bar{1}}}\mathrm{Id}|_{\mathfrak{g}_{\bar{1}}}
                                            =\frac{2n+1}{4n}\mathrm{Id}|_{\mathfrak{g}_{\bar{1}}},\\
&&\mathcal{C}_{\mathbf{B}|_{\mathfrak{k}_{2}}}=-\frac{n+1}{n}\frac{\mathrm{dim}\,\mathfrak{k}_{2}}{\mathrm{dim}\,\mathfrak{g}_{\bar{1}}}\mathrm{Id}|_{\mathfrak{g}_{\bar{1}}}
                                              =-\frac{2n+1}{4n}\mathrm{Id}|_{\mathfrak{g}_{\bar{1}}}.
\end{eqnarray*}
Equation (\ref{Einstein equiivalent}) becomes
\begin{eqnarray*}
\left\{
\begin{aligned}
&\frac{1}{4}(x_{1}^{2}-1)=c x_{1},\\
&\frac{1}{4}(x_{2}^{2}-1)=-\frac{n}{n+1}c x_{2},\\
&\frac{2n+1}{8n}(x_{1}-x_{2})=c.
\end{aligned}\right.
\end{eqnarray*}
There are four real solutions $(x_{1},x_{2},c)=\pm(1,1,0)$ and
$$\pm\left(
                  \frac{\sqrt{2}n}{\sqrt{2n^{2}+2n+1}},
                   \frac{\sqrt{2}(n+1)}{\sqrt{2n^{2}+2n+1}},
                    -\frac{\sqrt{2}(2n+1)}{8n\sqrt{2n^{2}+2n+1}}
                  \right).$$

7. Case $\mathbf{D}(2,1;\alpha)$,  $\alpha\in \mathbb{C}\setminus\{0,-1\}$.
We have
$\mathfrak{g}_{\bar{0}}=\mathfrak{so}(4,\mathbb{C})\oplus \mathfrak{sl}(2,\mathbb{C})=\mathbf{A}_{1}\oplus \mathbf{A}_{1}\oplus \mathbf{A}_{1}$, $\mathrm{dim}\,\mathfrak{k}_{1}=\mathrm{dim}\,\mathfrak{k}_{2}=\mathrm{dim}\,\mathfrak{k}_{3}=3$, $\mathrm{dim}\,\mathfrak{g}_{\bar{1}}=8$, $l_{1}=l_{2}=l_{3}=1$.
Define $\mathbf{B}$  on $\mathbf{D}(2,1;\alpha)$ by  $\mathbf{B}(X,Y)=2\,\mathrm{str}\,(XY)$.
 Then $\mathbf{B}|_{\mathfrak{k}_{1}}=\mathbf{K}_{1}$, $\mathbf{B}|_{\mathfrak{k}_{2}}=\mathbf{K}_{2}$, $\mathbf{B}|_{\mathfrak{k}_{3}}=-\frac{1}{2}\mathbf{K}_{3}$, so $b_{1}=b_{2}=1$, $b_{3}=-\frac{1}{2}$. We have
\begin{eqnarray*}
&&\mathcal{C}_{\mathbf{B}|_{\mathfrak{k}_{1}}}=\mathcal{C}_{\mathbf{B}|_{\mathfrak{k}_{2}}}=\frac{3}{8}\mathrm{Id}|_{\mathfrak{g}_{\bar{1}}},\\
&&\mathcal{C}_{\mathbf{B}|_{\mathfrak{k}_{3}}}=(-2)\frac{\mathrm{dim}\,\mathfrak{k}_{3}}{\mathrm{dim}\,\mathfrak{g}_{\bar{1}}}\mathrm{Id}|_{\mathfrak{g}_{\bar{1}}}
                                              =-\frac{3}{4}\mathrm{Id}|_{\mathfrak{g}_{\bar{1}}}.
\end{eqnarray*}
Equation (\ref{Einstein equiivalent}) becomes
\begin{eqnarray*}
\left\{
\begin{aligned}
&\frac{1}{4}(x_{1}^{2}-1)=c x_{1},\\
&\frac{1}{4}(x_{2}^{2}-1)=c x_{2},\\
&\frac{1}{4}(x_{3}^{2}-1)=-\frac{1}{2}c x_{3},\\
&\frac{3}{16}(x_{1}+x_{2}-2x_{3})=c.
\end{aligned}\right.
\end{eqnarray*}
There are four real solutions
$(x_{1},x_{2},x_{3},c)=\pm(1,1,1,0), \pm\left(\sqrt{\frac{2}{5}},\sqrt{\frac{2}{5}},2\sqrt{\frac{2}{5}},-\frac{3}{8}\sqrt{\frac{2}{5}}\right).$

8. Case $\mathbf{F}(4)$:
$$\mathfrak{g}=\mathbf{F}(4),\quad
\mathfrak{g}_{\bar{0}}=\mathbf{B}_{3}\oplus \mathbf{A}_{1},\quad \mathfrak{g}_{\bar{0}}|\mathfrak{g}_{\bar{1}}=\mathfrak{spin}(7)\otimes\mathfrak{sl}(2).$$
We have $\mathrm{dim}\,\mathfrak{k}_{1}=21$, $\mathrm{dim}\,\mathfrak{k}_{2}=3$, $\mathrm{dim}\,\mathfrak{g}_{\bar{1}}=16$, $l_{1}=\frac{2}{5}$, $l_{2}=2$. Then
\begin{eqnarray*}
\frac{l_{1}\mathrm{dim}\,\mathfrak{k}_{1}}{(1-l_{1})\mathrm{dim}\,\mathfrak{g}_{\bar{1}}}=\frac{7}{8},\quad
\frac{l_{2}\mathrm{dim}\,\mathfrak{k}_{2}}{(1-l_{2})\mathrm{dim}\,\mathfrak{g}_{\bar{1}}}=-\frac{3}{8}.
\end{eqnarray*}
Let $\mathbf{B}=\mathbf{K}$ and then equation \eqref{B=K equivalent} becomes
\begin{eqnarray*}
\left\{
\begin{aligned}
&\frac{1}{4}\left(\frac{2}{5}x_{1}^{2}-1\right)=\frac{3}{5}c x_{1},\\
&\frac{1}{4}(2 x_{2}^{2}-1)=-c x_{2},\\
&\frac{7}{8}x_{1}-\frac{3}{8}x_{2}=2c+1.
\end{aligned}\right.
\end{eqnarray*}
There is a unique real solution $(x_{1},x_{2},c)=\left(1,1,-\frac{1}{4}\right)$.

9. Case $\mathbf{G}(3)$:
$$\mathfrak{g}=\mathbf{G}(3),\quad
\mathfrak{g}_{\bar{0}}=\mathfrak{g}_{2}\oplus \mathbf{A}_{1},\quad
\mathfrak{g}_{\bar{0}}|\mathfrak{g}_{\bar{1}}=\mathfrak{g}_{2}\otimes\mathfrak{sl}(2).$$
We have  $\mathrm{dim}\,\mathfrak{k}_{1}=14$, $\mathrm{dim}\,\mathfrak{k}_{2}=3$, $\mathrm{dim}\,\mathfrak{g}_{\bar{1}}=14$, $l_{1}=\frac{1}{2}$, $l_{2}=\frac{7}{4}$. Then
\begin{eqnarray*}
\frac{l_{1}\mathrm{dim}\,\mathfrak{k}_{1}}{(1-l_{1})\mathrm{dim}\,\mathfrak{g}_{\bar{1}}}=1,\quad
\frac{l_{2}\mathrm{dim}\,\mathfrak{k}_{2}}{(1-l_{2})\mathrm{dim}\,\mathfrak{g}_{\bar{1}}}=-\frac{1}{2}.
\end{eqnarray*}
Let $\mathbf{B}=\mathbf{K}$ and then equation \eqref{B=K equivalent} becomes
\begin{eqnarray*}
\left\{
\begin{aligned}
&\frac{1}{4}\left(\frac{1}{2}x_{1}^{2}-1\right)=\frac{1}{2}c x_{1},\\
&\frac{1}{4}\left(\frac{7}{4} x_{2}^{2}-1\right)=-\frac{3}{4}c x_{2},\\
&x_{1}-\frac{1}{2}x_{2}=2c+1.
\end{aligned}\right.
\end{eqnarray*}
There are two real solutions $(x_{1},x_{2},c)=(1,1,-\frac{1}{4})$ and
 $(x_{1},x_{2},c)\approx (1.1760,0.8767,0.1312).$


\end{document}